\documentclass[11pt]{article}
\usepackage{amsthm, amsmath, amssymb, amsfonts, url, booktabs, tikz, setspace, fancyhdr, enumerate}
\usepackage{comment}
\usepackage[margin = 1in]{geometry}

\usepackage[czech,english]{babel} 

\usepackage{hyperref}
\usepackage{cleveref}

\usepackage{pst-node}
\usepackage{auto-pst-pdf}
\usepackage{tikz-cd}

\definecolor{refkey}{gray}{.75}
\definecolor{labelkey}{gray}{.2}

\usepackage[textsize=scriptsize,colorinlistoftodos]{todonotes}

\newtheorem{theorem}{Theorem}[section]
\newtheorem{proposition}[theorem]{Proposition}
\newtheorem{lemma}[theorem]{Lemma}
\newtheorem{corollary}[theorem]{Corollary}
\newtheorem{conjecture}[theorem]{Conjecture}

\theoremstyle{definition}

\newtheorem{example}[theorem]{Example}

\theoremstyle{remark}

\newcommand{\Hom}{\mathrm{Hom}}

\newcommand{\RR}{\mathbb{R}}

\newcommand{\bb}{\mathbf{b}}

\newcommand{\ext}{\mathrm{ex}}
\newcommand{\HH}{\mathcal{H}}

\newcommand{\cP}{\mathcal{P}}

\newcommand{\e}{\epsilon}

\newcommand{\mockalph}[1]{}

\usepackage{hyperref}
\usepackage{tikz-cd}
\usepackage{tikz}

\usepackage{ucs}
\usepackage{tkz-tab}
\usepackage{caption}
\usepackage{latexsym}
\usepackage{subcaption}
\usepackage{tcolorbox}
\usepackage{graphicx}

\tikzset{
	vertex/.style={circle, fill=black, inner sep=1,  minimum size=1.5mm},
	vertical align/.style={
		baseline=(current bounding box.center)
	},
	position/.style args={#1:#2 from #3}{
		at=(#3.#1), anchor=#1+180, shift=(#1:#2)
	},
}
\tikzset{
	vertexred/.style={circle, fill=red, inner sep=0pt, minimum size=2mm, anchor=center},
	every label/.append style={rectangle, outer sep=3pt},
}

\tikzset{
	vertexblue/.style={circle, fill=blue, inner sep=0pt, minimum size=2mm, anchor=center},
	every label/.append style={rectangle, outer sep=3pt},
}

\title{On the extremal number of incidence graphs}
\author{
Jisun Baek\thanks{Department of Mathematics, Yonsei University,  South Korea. Email: {\tt baek\_jisun@yonsei.ac.kr}. Research supported by Samsung STF Grant SSTF-BA2201-02 and the National Research Foundation of Korea (NRF) Grants MSIT NRF-2022R1C1C1010300 and NRF-2021R1A2C1009639.
}
\and
David Conlon\thanks{Department of Mathematics, California Institute of Technology, United States. Email:{\tt dconlon@caltech.edu}. Research supported by NSF Awards DMS-2054452 and DMS-2348859.
}
\and
Joonkyung Lee\thanks{Department of Mathematics, Yonsei University,  South Korea. Email: {\tt joonkyunglee@yonsei.ac.kr}. Research supported by Samsung STF Grant SSTF-BA2201-02, the Yonsei University Research Fund 2023-22-0125 and the National Research Foundation of Korea (NRF) Grant MSIT NRF-2022R1C1C1010300.
}
}
\date{}

\begin{document}
\maketitle

\begin{abstract}
Given a graph $H$ and a natural number $n$, the extremal number $\ext(n, H)$ is the largest number of edges in an $n$-vertex graph containing no copy of $H$. In this paper, we obtain a general upper bound for the extremal number of generalised face-incidence graphs, a family which includes the standard face-incidence graphs of regular polytopes. This builds on and generalises work of Janzer and Sudakov, who obtained the same bound for hypercubes and bipartite Kneser graphs, and allows us to confirm a conjecture of Conlon and Lee on the extremal number of $K_{r,r}$-free bipartite graphs for certain incidence graphs. 

In their work, Janzer and Sudakov showed that such an upper bound on the extremal number holds whenever the graph $H$ satisfies a certain percolation property which captures an appropriate sequence of repeated applications of the Cauchy--Schwarz inequality, a property which they then verify for hypercubes and bipartite Kneser graphs. This percolation property bears close resemblance to a property that arose in earlier work of Conlon and Lee on weakly norming graphs. In this latter work, Conlon and Lee developed a method for controlling repeated applications of the Cauchy--Schwarz inequality based on the properties of reflection groups, which then allowed them to isolate a broad family of weakly norming graphs. Here, we develop this method further, casting it in a purely algebraic form that allows us not only to combine it with the Janzer--Sudakov result and obtain the desired result about the extremal number of incidence graphs, but also to simplify the proofs of both the Conlon--Lee result on weakly norming graphs and a related result of Coregliano.
\end{abstract}

\section{Introduction}

The Cauchy--Schwarz inequality is often used in extremal combinatorics to bound the number of discrete structures of a given type. Its use is so pervasive that several techniques have been developed that allow one to automate, in some appropriate sense, repeated applications of the inequality. One example of this is the flag algebra technique developed by Razborov~\cite{R07}, which allows one to find such sequences of inequalities computationally through the use of semidefinite programming. This is arguably one of the most important developments in extremal combinatorics this century and has allowed researchers to make significant progress on a range of notorious problems (see, for example,~\cite{BT11, HHKNR12, R08}).

Our focus in this paper is on another method for automating repeated applications of the Cauchy--Schwarz inequality developed by Conlon and Lee~\cite{CL16}. This method, which we call the \emph{reflection group method}, applies only to graphs of a certain type, which we call reflection graphs, but, for those graphs, applications of the Cauchy--Schwarz inequality can be encoded algebraically in terms of reflections within an appropriate finite reflection group, which allows one to analyse them in an efficient way. This method has proved surprisingly useful for studying problems on graph homomorphism inequalities, such as Sidorenko's conjecture, where it has led to a proof that the conjecture holds for large enough blow-ups of any fixed bipartite graph~\cite{CL20}, and a question of Lov\'asz and Hatami on categorising graph norms~\cite{H10,L08}, where it has suggested a possible classification~\cite{CL16}.

To get a better idea of how the reflection group method works, it will help to describe how the Cauchy--Schwarz inequality can be applied to graphs. A key role in applying Cauchy--Schwarz in this context is played by a particular type of automorphism that we call a cut involution, which is essentially a reflection of the graph into itself. 
Formally, a \emph{cut involution} of a graph $H$ is an involutary automorphism $\phi$ paired with a tripartition $L\cup F\cup R$ of the vertex set $V(H)$ such that $\phi$ swaps vertices in $L$ and $R$ while fixing all the vertices in the vertex cut $F$. An application of the Cauchy--Schwarz inequality then gives that
\begin{align*}
    t(H,G)\cdot t(H[F],G) \geq t(H[L\cup F],G)^2, 
\end{align*}
where $t(K,G)$ is the \emph{homomorphism density} $|\Hom(K,G)|/|V(G)|^{|V(K)|}$ of $K$ in $G$ and $H[U]$ is the induced subgraph of $H$ on $U\subseteq V(H)$. That is, cut involutions of $H$ give rise, through the Cauchy--Schwarz inequality, to graph homomorphism inequalities between its subgraphs. For any reflection graph, a family which we will define formally in~\Cref{sec:prelim} below, there is an embedding of the graph in some Euclidean space such that cut involutions correspond to geometric reflections across hyperplanes through the origin. In turn, this allows us to leverage the properties of finite reflection groups when studying sequences of graph homomorphism inequalities between subgraphs of reflection graphs.

The main result of this paper is an application of the reflection group method to the study of extremal numbers. Given a graph $H$ and a natural number $n$, the {\it extremal number} $\ext(n, H)$ is the largest number of edges in an $n$-vertex $H$-free graph. While the classical Erd\H{o}s--Stone--Simonovits theorem determines the asymptotic behaviour of $\mathrm{ex}(n,H)$ for non-bipartite $H$, determining the asymptotics of $\mathrm{ex}(n,H)$ for bipartite $H$ remains an important and difficult open problem, one that has been solved for only a handful of examples.

Quite recently, a novel approach to the study of extremal numbers of bipartite graphs, making use of graph homomorphism inequalities, has been developed by Janzer and Sudakov~\cite{JS24} and, independently, Kim, Lee, Liu and Tran~\cite{KLLT24}. To give a rough sketch of the idea, note that if an $n$-vertex graph $G$ is $H$-free for a particular bipartite graph $H$, then all homomorphic copies of $H$ must be \emph{degenerate}, in the sense that at least two vertices must map to the same vertex. The key observation of~\cite{JS24,KLLT24} is that, provided $H$ has a particular form, repeated applications of the Cauchy--Schwarz inequality allow us to conclude that there are many degenerate copies where more and more vertices coincide until, eventually, we arrive at the fact that there are many degenerate copies of $H$ where all of the vertices on one side of its bipartition agree. That is, there are many stars of a particular size in $G$. We can then trade this off against an easily obtained upper bound for the number of such stars to obtain an upper bound for $\ext(n, H)$. 

In their paper~\cite{JS24}, Janzer and Sudakov showed that their process applies to any graph $H$ which satisfies a certain combinatorially defined percolation property and, moreover, that this property is satisfied by hypercubes and by bipartite Kneser graphs, allowing them to improve the upper bounds for the extremal numbers of these graphs. In particular, their bounds confirm the following conjecture of Conlon and Lee~\cite{CL21} for these graphs.

\begin{conjecture}\label{conj:CL}
    If $H$ is a $K_{r,r}$-free bipartite graph with degree at most $r$ on one side, then there exists $c=c_H>0$ such that $\ext(n,H)=O(n^{2-1/r-c})$.
\end{conjecture}

The motivation behind this conjecture comes from a classical result of F\"uredi~\cite{Fu91} (see also~\cite{AKS03}) saying that if $H$ is a bipartite graph with degree at most $r$ on one side, then $\ext(n,H)=O(n^{2-1/r})$. By a celebrated result of Koll\'ar, R\'onyai and Szab\'o~\cite{KRSz96}, this bound is known to be tight when $H = K_{r,s}$ for $s$ sufficiently large in terms of $r$ (see~\cite{Bu24} for a recent improvement on the quantitative aspects of this theorem) and it is conjecturally tight when $H = K_{r,r}$, though this is somewhat controversial. What Conjecture~\ref{conj:CL} is saying is that the only reason F\"uredi's bound should be tight is if $H$ contains $K_{r,r}$. For $r = 2$, this conjecture was proved in the original paper of Conlon and Lee~\cite{CL21} and an elegant alternative proof was subsequently found by Janzer~\cite{Ja19}. More recently, a weak version of the general conjecture stated in~\cite{CJL21}, asking for $\ext(n,H)=o(n^{2-1/r})$ instead of $\ext(n,H)=O(n^{2-1/r-c})$, was proved by Sudakov and Tomon~\cite{ST20}.

It was already noted in~\cite{JS24} that there is a similarity between the percolation property they need and a percolation property shown in an earlier paper of Conlon and Lee~\cite{CL16} to imply that a graph is weakly norming, which roughly means that there are a collection of very strong homomorphism inequalities between the graph and its subgraphs. The reflection group method was developed in that same paper~\cite{CL16} to show that a certain family of graphs, the aforementioned reflection graphs, satisfies the required percolation property. As such, it is natural to ask whether the reflection group method also applies to the percolation property used by Janzer and Sudakov. We show that the answer is yes, using the reflection group method to show that a broad family of reflection graphs satisfy their percolation property. We may then use this and their main result~\cite[Theorem~2.16]{JS24} to give an upper bound for the extremal numbers of these graphs. 

To state one such result, let~$\mathcal{P}$ be a regular polytope and, for $k<r$, let $H$ be the bipartite graph between the set of all $k$-faces and the set of all $r$-faces of $\mathcal{P}$, where an $r$-face and a $k$-face are adjacent if and only if one contains the other. We call this graph the \emph{$(k,r)$-incidence graph} of $\cP$. For instance, the bipartite Kneser graph which has an edge between a $k$-subset and a $(d-k)$-subset of $[d]$ if and only if one includes the other is the $(k-1,d-k-1)$-incidence graph of a $(d-1)$-dimensional simplex. 

\begin{theorem}\label{thm:incidence_main}
    Let $H$ be the $(k,r)$-incidence graph of a regular polytope $\cP$ and let $t$ be the order of the largest part in its bipartition. Then
    $\mathrm{ex}(n,H) =O(n^{2-c})$, where $c = \frac{v(H)-t-1}{e(H)-t}$.
\end{theorem}

The family of $(k,r)$-incidence graphs of regular polytopes does not include the hypercubes, but this latter family is included in a slightly more general class, which we call generalised face-incidence graphs, to which our results also apply. Since the precise definition of these graphs requires some notation on reflection groups, we will hold back on stating it formally until the next section. 

Besides our chief purpose of extending the results of Janzer and Sudakov, we have also taken the opportunity to explain the reflection group method in a more accessible manner. In the original paper~\cite{CL16}, Conlon and Lee took a slight detour, embedding reflection graphs into Euclidean space before casting them in a purely algebraic setting. Here we skip this step, showing that one can jump straight to the algebraic setting. As a consequence, we are able to give a somewhat simpler proof for the main result in~\cite{CL16}, as well as for another related result of Coregliano~\cite{C24}.

The remainder of the paper is organised as follows. In the next section, we give a brief introduction to finite reflection, or Coxeter, groups. In particular, we will be able to define both reflection graphs and generalised face-incidence graphs, thus allowing us to state our real main result, \Cref{thm:main}. We also note a range of useful results about Coxeter groups, culminating in a result saying that reflections in Coxeter groups yield cut involutions in the corresponding reflection graphs. \Cref{sec:percseq} contains the proofs of our main results. We first warm up by giving a short proof of the result from~\cite{CL16} saying that reflection graphs are weakly norming. We then proceed to the proof of~\Cref{thm:main}, after which we spend some time studying specific examples and analysing when they satisfy~\Cref{conj:CL}. We conclude with some brief further remarks.

\section{Preliminaries}\label{sec:prelim}

A \emph{finite reflection group}, usually denoted by the letter $W$, is a finite subgroup of the general linear group $\mathbf{GL}(n,\RR)$ generated by reflections across hyperplanes passing through the origin.
The set of all reflections in $W$, denoted by $T$, then forms a subset of $W$. 
Among these reflections, one may identify particular generating sets $S$, known as sets of \emph{simple reflections}, that play a central role in the theory. There is a geometric way to view these sets of simple reflections, but we shall discuss them in purely algebraic terms. When interpreted in this way, reflection groups are often referred to as \emph{Coxeter groups}.

Formally, a \emph{Coxeter group} $W$ is a group with presentation $\langle s_1,s_2,\dots,s_n \vert (s_is_j)^{m_{ij}}=1\rangle$, where $m_{ii}=1$ for all $i$ and $m_{ij}=m_{ji}$ is an integer greater than $1$ or $\infty$.\footnote{$m_{ij}=\infty$ means that there is no relation between $s_i$ and $s_j$, a case which never occurs for $W$ finite.} The pair $(W, S)$ where $W$ is a Coxeter group with $S=\{s_1,\dots,s_n\}$ its set of generators is then called a \emph{Coxeter system}. 
For us, a Coxeter system will consist of a finite reflection group, which, by the fundamental work of Coxeter~\cite{C34}, is isomorphic to a finite Coxeter group $W$, together with a generating set $S$ corresponding to a set of simple reflections in that finite reflection group. Because of this correspondence, given a Coxeter system $(W,S)$, we will refer to the elements $s\in S$ as \emph{simple reflections}. Throughout what follows, $(W,S)$ will always denote such a Coxeter system. 

We now come to the definition of reflection graphs, an algebraically-defined class of graphs to which we can apply the reflection group method. 
For $I\subseteq S$, let $W_I$ be the subgroup of $W$ generated by $I$, which is usually referred to as a \emph{parabolic subgroup}. Let $W/W_I$ denote the set of all left cosets of $W_I$.
For $I,J\subseteq S$, the \emph{$(I,J;W,S)$-graph} is the bipartite graph between $W/W_I$ and $W/W_J$ with $wW_I$ and $wW_J$ adjacent for all $w \in W$. That is, two cosets are adjacent if and only if they intersect. A \emph{reflection graph} is now any graph that is isomorphic to the $(I,J;W,S)$-graph for some $(W,S)$ and $I,J\subseteq S$. 

A well-known fact from the theory of reflection groups (see, for example,~\cite[Theorem~3D7]{MS02} and~\cite[Theorem~1.1]{CL16}) is that the $(k,r)$-incidence graph of a regular polytope $\cP$ is isomorphic to an $(I,J;W,S)$-graph, where $W$ is the underlying symmetry group of $\cP$ and $|I|=|J|=|S|-1$.
More generally, we use the term \emph{generalised face-incidence graph} to refer to any graph which is isomorphic to a $(I,J;W,S)$-graph with $|I|=|J|=|S|-1$. As noted in the introduction, hypercubes are not face-incidence graphs of any regular polytopes, but they are generalised face-incidence graphs (see~\cite[Example 4.14]{CL16}), so this latter family is indeed somewhat more general. Our main result, which therefore subsumes~\Cref{thm:incidence_main}, is now as follows.

\begin{theorem}\label{thm:main}
    Let $H$ be a generalised face-incidence graph and let $t$ be the order of the largest part in its bipartition. Then
    $\mathrm{ex}(n,H) =O(n^{2-c})$, where $c = \frac{v(H)-t-1}{e(H)-t}$.
\end{theorem}

The proof of this result has two main steps. First, we prove that every reflection in $W$ naturally defines a cut involution of the $(I,J;W,S)$-graph, no matter what $I$ and $J$ are. In fact, we will prove a slightly more general statement concerning $r$-uniform hypergraphs, which we will henceforth refer to as \emph{$r$-graphs}. 
To describe the context of this more general result, let $I_1,\dots,I_r$ be subsets of $S$. For brevity, we shall write $W_i:=W_{I_i}$ in what follows. Then the \emph{$(I_1,\dots,I_r;W,S)$-graph} is the $r$-partite $r$-graph on $(W/W_1)\cup (W/W_2)\cup\cdots\cup (W/W_r)$ whose edges are all $(wW_{1},\dots,wW_{r})$ with $w \in W$. A \emph{cut involution} of an $r$-graph $\HH$ is an involutary automorphism of $\HH$ paired with a tripartition $L\cup F\cup R$ of $V(\HH)$ that swaps $L$ and $R$ and where the fixed point set $F$ is a \emph{vertex cut} in the sense that no $r$-edge of $\HH$ intersects both $L$ and $R$. 

To state our result about how reflections define cut involutions, whose proof will occupy the rest of \Cref{sec:prelim}, we also need to formally define reflections in Coxeter groups. We have already noted that the generating set $S$ in a finite Coxeter system can be identified with a set of simple reflections in a corresponding finite reflection group. The  reflections in that reflection group correspond exactly to the elements of the set $T := \{wsw^{-1} : s \in S, w \in W\}$, 
which we now refer to as \emph{reflections}. It is easy to verify that $S\subseteq T$ and $t^2=e$ for every $t\in T$. 

\begin{theorem}\label{thm:cut_involution}
    Let $t\in T$ be a reflection in $W$. Then the map $\phi_t(wW_i):=twW_i$ is a cut involution of the $(I_1,\dots,I_r;W,S)$-graph, where $W_i=W_{I_i}$ is the parabolic subgroup generated by $I_i$.
\end{theorem}

Second, by using this fact, we will translate certain `combinatorial' percolation processes on reflection graphs to an algebraic setting, thereby allowing us to apply the reflection group method and prove our main results, though the precise details will vary depending on the extremal problem we wish to consider. 
Those readers who are only interested in these applications may skip the proof of~\Cref{thm:cut_involution} and proceed straight to \Cref{sec:percseq}.

\subsection{Basic properties of Coxeter systems}

In this subsection, we give a brief overview of some of the basic properties of Coxeter systems, focusing on those properties that are needed for the proof of~\Cref{thm:main}. For a more thorough introduction, we refer the reader to either \cite{BB05} or \cite{H92}. 

The \emph{length} $\ell(w)$ of an element $w\in W$ is the minimum word length taken over all expressions for $w$ in terms of simple reflections, i.e., the minimum $k$ such that $w=s_1s_2\cdots s_k$ for some $s_i\in S$, $i=1,\dots,k$. A minimum-length expression for $w$ in terms of simple reflections is called a \emph{reduced word} or \emph{reduced expression} for~$w$, though we stress that this may not be unique.
It is an elementary fact that the parity of every expression for $w$ is the same.

\begin{lemma}[Lemma~1.4.1 in~\cite{BB05}]\label{parity}
    The map $w\mapsto (-1)^{\ell(w)}$ is a group homomorphism from $W$ to the multiplicative group $\{+1,-1\}$.
    In particular, every reflection $t=wsw^{-1}$, for $s\in S$ and $w\in W$, has odd length, so $\ell(w)\neq \ell(tw)$ mod $2$ for all $w \in W$ and $t \in T$.
\end{lemma}

We now state a fundamental combinatorial property of Coxeter systems, the so-called \emph{strong exchange property}. It states that whenever the length of $w$ is reduced by multiplying by a reflection $t\in T$ on the left, each expression $w=s_1s_2\cdots s_k$ reduces to one that deletes some $s_i$ while the order of the expression is preserved. If $s_i$ is deleted from $s_1s_2\cdots s_k$, then we simply write the resulting expression $s_1\cdots s_{i-1}s_{i+1}\cdots s_k$ as $s_1 \cdots \hat{s_i} \cdots s_k$. We remark that the proof relies on~\Cref{parity}.

\begin{theorem}[Theorem~1.4.3 in~\cite{BB05}]\label{thm:strongexchange}
    Let $w=s_1 s_2 \cdots s_k$, $s_i \in S$, and let $t \in T$. If $\ell(tw) < \ell(w)$, then $tw=s_1 \cdots \hat{s_i} \cdots s_k$ for some $i =1,2,\dots,k$.
    That is, $t=us_iu^{-1}$, where $u=s_1s_2\cdots s_{i-1}$.
\end{theorem}

Another fundamental property of Coxeter systems is the \emph{deletion property}. This says that one can reduce any unreduced expression by deleting exactly two simple reflections while preserving the order of multiplications.

\begin{proposition} [Proposition~1.4.7 in~\cite{BB05}]\label{thm:deletion}
    If $w=s_1 s_2 \cdots s_k$ and $\ell(w)<k$, then, for some $1 \leq i < j \leq k$, $w=s_1 \cdots \hat{s_i}\cdots \hat{s_j} \cdots s_k$.
\end{proposition}

These two properties characterise Coxeter systems, in the sense that a group $W$ generated by a set of involutions $S$ is a Coxeter system if and only if it has the deletion property (or the strong exchange property) with respect to $S$ (see Theorem~1.5.1 in~\cite{BB05} for the proof). 

We will need the following corollary of~\Cref{thm:deletion}.

\begin{corollary}[Corollary~1.4.8 in~\cite{BB05}]\label{cor:generating}
    Any reduced expression for $w$ contains the same set of simple reflections.
\end{corollary}

In the proof of~\Cref{thm:main}, we will rely heavily on a partial order defined on $W$ and its quotients, often referred to as the \emph{Bruhat order}. We have that $u < w$ in this order if there exists a sequence $u_0,\dots,u_{k} \in W$ such that $u_0=u$, $u_{k}=w$, $u_{i+1}=u_i t_i$,\footnote{This definition is equivalent to the left-multiplication version, since such a sequence satisfies $u_{i+1}=t_i' u_i$ for $t_i'=u_it_i{u_i}^{-1}$. Therefore, $u<w$ is equivalent to $u^{-1} < w^{-1}$.} and $\ell(u_i)<\ell(u_{i+1})$ for each $i=0,1,\dots,k-1$ with $t_i\in T$. The notation $u\leq w$ then means that either $u=w$ or $u<w$. When $u$ and $w$ are both written in reduced form, $u\leq w$ simply means that deleting some simple reflections from $w$ gives $u$, i.e., $u$ can be seen as a subword of $w$. This is called the \emph{subword property}.

\begin{theorem}[Theorem~2.2.2 and Corollary 2.2.3 in~\cite{BB05}]
    Let $w = s_1 s_2 \cdots s_q$ be a reduced expression. Then $u\leq w$ if and only if $u=s_{i_1} s_{i_2} \cdots s_{i_k}, 1 \leq i_1 < \cdots < i_k \leq q$. That is, $u\leq w$ if and only if every reduced expression for $w$ has a subword which is a reduced expression for $u$.
\end{theorem}

Another useful fact is the \emph{chain property} of the Bruhat poset.

\begin{theorem}[Theorem~2.2.6 in~\cite{BB05}]
    If $u<w$, there exists a chain $u=x_0<x_1 < \cdots < x_k = w$ such that $\ell(x_i) = \ell(u)+i$ for $1 \leq i \leq k$.
\end{theorem}

We end this subsection by stating the \emph{lifting property}, which will be needed for the proof of \Cref{thm:cut_involution}. 
To state the result, we need the notation $D_L(w) :=\{ s \in S : \ell(sw)<\ell(w)\}$. That is, $D_L(w)$ consists of those simple reflections that decrease the length of $w$ when we multiply on the left. Similarly, we write $D_R(w):=\{ s \in S : \ell(ws)<\ell(w)\}$.

\begin{proposition}[Proposition~2.2.7 in~\cite{BB05}]\label{lifting}
Suppose $u <w$ and $s \in D_L(w) \setminus D_L(u)$. Then $u \leq sw$ and $su \leq w$. By replacing $u$ and $w$ by $u^{-1}$ and $w^{-1}$, respectively, $u<w$ and $s \in D_R(w) \setminus D_R(u)$ implies $u \leq ws$ and $us \leq w$.
\end{proposition}

\subsection{Quotients of parabolic subgroups and the proof of~\Cref{thm:cut_involution}}

Let $W^I := \{w \in W : ws > w \text{ for all } s \in I\}$. This is called the \emph{quotient} of the parabolic subgroup $W_I$ generated by $I\subseteq S$. Indeed, as implied by the next result, the elements of $W^I$ are exactly the representatives of each left coset of $W_I$ that have minimum length.

\begin{proposition}[Proposition~2.4.4 in~\cite{BB05}]\label{thm:factorize}
    Let $I \subseteq S$. Then every $w \in W$ has a unique factorization $w = w^I \cdot w_I$ such that $w^I \in W^I$ and $w_I \in W_I$. Furthermore, $\ell(w)=\ell(w^I)+\ell(w_I)$.
\end{proposition}

A useful fact is that the projection $w\mapsto w^I$, denoted by $P^I$, preserves the Bruhat order.

\begin{proposition}[Proposition~2.5.1 in~\cite{BB05}]\label{prop:order_preserving}
    Let $u,w\in W$. If $w\leq u$, then $w^I\leq u^I$.
\end{proposition}

The chain property also carries over to quotients.

\begin{theorem}[Theorem~2.5.5 in~\cite{BB05}]\label{thm:reduced_subwords}
If $u < w$ in $W^I$, there exists a chain $u = w_0 < w_1 < \cdots < w_k=w$ such that $w_i \in W^I$ and $\ell(w_i) = \ell(u)+i$ for $1 \leq i \leq k$.
\end{theorem}

Recall that the map $\phi_t$ in~\Cref{thm:cut_involution} is defined by $\phi_t(wW_I)=twW_I$ for $I\subseteq S$. The next lemma shows that the set $F=F(I,t)$ of cosets $wW_I$ that are fixed under $\phi_t$ corresponds exactly to the set of $w$ such that $\ell((tw)^I)=\ell(w^I)$.

\begin{lemma}\label{fixed points}
Let $W_I$ be the parabolic subgroup generated by $I\subseteq S$. Then $\phi_t(wW_I)=wW_I$ if and only if $\ell((tw)^I)=\ell(w^I)$.
\end{lemma} 

\begin{proof}
As $twW_I=wW_I$ implies $(tw)^I=w^I$, $\ell((tw)^I)=\ell(w^I)$ trivially follows. It therefore remains to prove the converse. 

Suppose that $\ell((tw)^I)=\ell(w^I)$ holds. We may assume the representative $w$ already has the minimum length, i.e., $w=w^I$.
By~\Cref{parity}, multiplying by a reflection changes the parity, so that $\ell(tw)$ and $\ell(w)$ are distinct. 
As $(tw)^I$ has the minimum length amongst all elements in $twW_I$,
\begin{align*}
    \ell(w)=\ell(w^I)=\ell((tw)^I) \leq \ell(tw),
\end{align*}
which gives $w\leq tw$. Since $w\neq tw$, we must have that $w < tw$. 

We claim that if $\ell(u^I)=\ell(w^I)=\ell(w)$ and $w<u$, then $uW_I=wW_I$. If so, taking $u=tw$ concludes the proof.
Since $\ell(u^I)=\ell(w^I)=\ell(w)$ and $w<u$, \Cref{thm:factorize} implies that $\ell(u_I)=\ell(u)-\ell(w)$ is positive. 
Thus, again by~\Cref{thm:factorize}, $u=u^I\cdot u_I$ with nontrivial~$u_I$. By choosing $s\in I$ in the rightmost position of a reduced word for $u_I$, we have $us<u$. 
On the other hand, since $w=w^I$, $ws>w$, which means $s$ is in $D_R(u)\setminus D_R(w)$. 
Hence, the lifting property,~\Cref{lifting}, implies that $us\geq w$.
If $us =w$, then $uW_I=usW_I=wW_I$. Otherwise, $us > w$, but we still have $\ell((us)^I) = \ell(u^I) = \ell(w)$. 
Repeating the same process by replacing $u$ by $us$ until we get $us_1 \cdots s_n = w$ proves the claim.
\end{proof}

We are now in a position to prove~\Cref{thm:cut_involution}, which we recall states that the map $\phi_t(wW_i)=twW_i$ is a cut involution of the $(I_1,\dots,I_r;W,S)$-graph, where $W_i$ is a shorthand for $W_{I_i}$.

\begin{proof}[Proof of~\Cref{thm:cut_involution}]
For each $i=1,2,\dots, r$, let
\begin{align*}
    L_i & :=\{wW_{i} : \ell((tw)^{I_i}) > \ell(w^{I_i})\},\\
    R_i & :=\{wW_{i} : \ell((tw)^{I_i}) < \ell(w^{I_i})\},\\ 
    F_i & :=\{wW_{i} : \ell((tw)^{I_i}) = \ell(w^{I_i})\}.
\end{align*}
Set $L:=\cup_{i=1}^{r} L_i$, $R:=\cup_{i=1}^{r} R_i$ and $F:=\cup_{i=1}^{r} F_i$.

We claim that $\phi_t$ is a cut involution such that $L$ and $R$ are mapped to each other and $F$ is fixed.
Firstly, $\phi_t$ is a homomorphism since it maps each edge $(wW_1,\dots,wW_r)$ to another edge $(twW_1,\dots,twW_r)$.
The fact that $F$ is fixed under $\phi_t$ follows from~\Cref{fixed points}.
If $wW_i\in L_i$, then $\ell((tw)^{I_i})>\ell(w^{I_i})$, so $u=tw$ satisfies $\ell(u^{I_i})>\ell((tu)^{I_i})$. This simply means that $uW_i\in R_i$. Since, similarly, $uW_i\in L_i$ if $wW_i\in R_i$, $L_i$ and $R_i$ are mapped to each other under $\phi_t$.

Finally, suppose that there is an $r$-edge $(wW_1,\dots,wW_r)$ of the $(I_1,\dots,I_r;W,S)$-graph that intersects both $L$ and $R$. That is, 
$\ell((tw)^I)>\ell(w^I)$ and $\ell((tw)^J)<\ell(w^J)$ for some $I=I_i$ and $J=I_j$. But this contradicts~\Cref{prop:order_preserving}, which implies that both of the projection maps $P^I : W \rightarrow W^I$ and $P^J:W \to W^J$ are order-preserving.
\end{proof}

In what follows, the cut involution $\phi_t$ of the $(I_1,\dots,I_r;W,S)$-graph that corresponds to a reflection $t\in T$ will always be oriented in a way that is consistent with the proof of~\Cref{thm:cut_involution}. That is, $L:=\cup_{i=1}^{r} L_i$, $R:=\cup_{i=1}^{r} R_i$ and $F:=\cup_{i=1}^{r} F_i$. 

Finally, we note the following intersection property that will allow us to encode the edges of the $(I_1,\dots,I_r;W,S)$-graph as cosets of the parabolic subgroup $W_I$ with $I = \cap_{i=1}^r I_i$.

\begin{proposition}[Proposition~2.4.1 in~\cite{BB05}]\label{prop:intersecting_parabolic}
    Let $W_I$ and $W_J$ be the parabolic subgroups generated by $I, J \subseteq S$. Then $W_{I\cap J}=W_I\cap W_J$. That is, the parabolic subgroup generated by $I\cap J$ is the intersection of the parabolic subgroups generated by $I$ and $J$.
\end{proposition}

\section{Percolating sequences in Coxeter groups} \label{sec:percseq}

Let $\HH$ be an $r$-graph and let $\phi$ be a cut involution of $\HH$. The \emph{left-folding map} $\phi^{+} : V(\HH) \rightarrow V(\HH)$ is defined by
\begin{align*}
    \phi^{+}(v)= 
    \begin{cases}
    \phi(v) & \text{if} \: v \in R \\
    v & \text{if} \: v \in L \cup F.
    \end{cases}
\end{align*}
The \emph{right-folding map} $\phi^{-}$ is defined in a symmetric way by swapping $L$ and $R$. As folding maps are homomorphisms on $\HH$, the image of an edge subset $J\subseteq E(\HH)$ under a folding map is well defined as an edge subset of $\HH$.

For a vertex subset $K\subseteq V(\HH)$, let
\begin{align*}
    K^{+}(\phi):= \{v \in V(\HH) : \phi^{+}(v) \in K\} \text{ and } K^{-}(\phi):= \{v \in V(\HH) : \phi^{-}(v) \in K\}.
\end{align*}
Equivalently, $K^{+}(\phi)= (K\cap(L \cup F))\cup \phi(K\cap L)$ and $K^{-}(\phi)= (K\cap(R\cup F))\cup \phi(K\cap R)$. That is, to obtain $K^+(\phi)$, we keep all elements of $K$ that are in $L$ or $F$ but replace $K \cap R$ by $\phi(K\cap L)$.
Similarly, for an edge subset $J\subseteq E(\HH)$, let 
\begin{align*}
    J^{+}(\phi) :=  \{(v_1, v_2, \dots, v_r) \in E(\HH) : (\phi^+(v_1), \phi^+(v_2), \dots, \phi^+(v_r)) \in J \} 
\end{align*}
and let $J^-(\phi)$ be defined by replacing $\phi^+$ by $\phi^-$.

The folding maps capture a particular type of application of the Cauchy--Schwarz inequality to graph homomorphism densities. If we let $J\subseteq E(\HH)$ be an edge subset of an $r$-graph $\HH$, which can also be viewed as a spanning subgraph of $\HH$, then the Cauchy--Schwarz inequality gives
\begin{align*}
    t(J,G) \leq t(J^+(\phi),G)^{1/2}t(J^-(\phi),G)^{1/2},
\end{align*}
in a sense allowing us to transform $J$ into either $J^+(\phi)$ or $J^-(\phi)$. 
Our goal here will be to `lift' these transformations to the algebraic setting. 

To this end, we now observe that, in the $(I_1,\dots,I_r;W,S)$-graph, vertices and edges both correspond to algebraic objects of the same type, namely, cosets of a parabolic subgroup. 

\begin{proposition}\label{prop:correspondence}
    Suppose $I_1, \dots, I_r$ are subsets of $S$ and $I = I_1\cap I_2\cap\cdots\cap I_r$. If $W_i = W_{I_i}$ for $1 \le i \le r$, then the map $e = (wW_1,\dots,wW_r) \mapsto wW_I$ is a bijection from the edges of the $(I_1,\dots,I_r;W,S)$-graph to the left cosets of $W_I$.
\end{proposition}

\begin{proof}
    Suppose that $(w W_1, \dots,wW_r) = (w' W_1, \dots,w'W_r)$.
    This is equivalent to $w^{-1}w' \in W_i$ for all $i=1,2,\dots,r$, which means that $wW_I=w'W_I$ by \Cref{prop:intersecting_parabolic}.
    Thus, the map $\varphi(w W_1, \dots,wW_r)=wW_I$ is well-defined.
    If $wW_I=w'W_I$, then $w^{-1}w'\in W_I$, so $(w W_1, \dots,wW_r) = (w' W_1, \dots,w'W_r)$, i.e., $\varphi$ is an injection. 
\end{proof}

This key observation allows us to take a purely algebraic view of folding maps. 
For a reflection $t$ in a Coxeter system $(W,S)$, let
\begin{align*}
     L_t:=\{w\in W : \ell(tw) > \ell(w)\}~\text{and}~R_t:=\{w\in W : \ell(tw) < \ell(w)\}.
\end{align*}
Note that $L_t\cup R_t=W$, since $\ell(tw)\neq \ell(w)$ by~\Cref{parity}. 
The \emph{left-folding map} $t^+$ is then defined by 
\begin{align*}
    t^{+}(w)= 
    \begin{cases}
    tw & \text{if} \: w \in R_t \\
    w & \text{if} \: w \in L_t.
    \end{cases}
\end{align*}
Symmetrically, the \emph{right-folding map} $t^-(w)= tw$ if $w\in L_t$ and $w$ otherwise. 

If $I\subseteq S$ is a set of simple reflections and $W_I$ is the parabolic subgroup generated by $I$, we define 
$$t^+(wW_I):=t^+(w)W_I.$$ 
This left-folding map on the cosets of $W_I$ also preserves, in a sense, the `direction' of the folding map. 
To see this, let
\begin{align*}
     L_t^I &:=\{wW_{I} : \ell((tw)^{I}) > \ell(w^{I})\},\\
     R_t^I &:=\{wW_{I} : \ell((tw)^{I}) < \ell(w^{I})\},\\
     F_t^I &:=\{wW_{I} : \ell((tw)^{I}) = \ell(w^{I})\}.
\end{align*}
Then $t^+$ satisfies
\begin{align*}
    t^{+}(wW_I)= 
    \begin{cases}
    twW_I & \text{if} \: wW_I \in R_t^I \\
    wW_I & \text{if} \: wW_I \in L_t^I \cup F_t^I.
    \end{cases}
\end{align*}
Indeed, since $P^I : W \mapsto W^I$ is order-preserving, $w \in L_t$ implies $wW_I \in L_t^I \cup F_t^I$ and $w \in R_t$ implies $wW_I \in R_t^I \cup F_t^I$. Moreover, $w \in R_t$ and $wW_I \in F_t^I$ implies that $twW_I = wW_I$, by \Cref{fixed points}. 
We can similarly define the \emph{right-folding map} $t^{-}$ on $W/W_I$, i.e., $t^-(wW_I)=t^-(w)W_I$. Then $t^-(wW_I)=twW_I$ if $wW_I\in L_t^I$ and $wW_I$ otherwise.

For $J\subseteq W/W_I$, let
\begin{align*}
    J^+(t) := \{wW_I:t^+(wW_I)\in J\} \text{ and } J^-(t) := \{wW_I:t^-(wW_I)\in J\}.
\end{align*}
A sequence of subsets $J_0, J_1, \dots$ of $W / W_I$ is now said to be a \textit{folding sequence} if there is a reflection $t_i$ such that $J_{i+1} = J_i^+(t_i)$ or $J_{i+1} = J_i^-(t_i)$.

Recall that \Cref{thm:cut_involution} shows that $\phi_t(v)=twW_I$ is a cut involution of the $(I_1,\dots,I_r;W,S)$-graph $H$, where $v$ is the vertex in $H$ that corresponds to $wW_I$ for some $I=I_j$. Furthermore, from the proof of~\Cref{thm:cut_involution}, $L_t^I=L_{\phi_t}$ and $R_t^I=R_{\phi_t}$. 
Therefore, $t^+(v)=\phi_t^+(v)$ and $t^+(e)=\phi_t^+(e)$ for each vertex $v$ and edge $e$ in the $(I_1,\dots,I_r;W,S)$-graph.
For a vertex or edge subset $J$ of the $(I_1,\dots,I_r;W,S)$-graph, it follows that $J^+(\phi_t)$ is the same as $J^+(t)$, provided we identify $J$ with a subset of $W/W_I$ for a suitable choice of $I$, i.e., $I=I_j$ if $J$ is a vertex subset of the $j$-th part $W/W_j$ and $I=\cap_{i=1}^{r}I_i$ if $J$ is an edge subset. A similar correspondence holds for $t^-$ and $\phi_t^-$, as well as for $J^-(\phi_t)$ and $J^-(t)$.

\subsection{Percolating sequences: the Conlon--Lee theorem}

As a warm-up, we reprove the main result (Theorem 1.2) of~\cite{CL16}, which states that every reflection graph is weakly norming. 
We will not say much about the weakly norming property here, as we prove the result using a reduction step from~\cite{CL16}, but informally a graph is weakly norming if the graph homomorphism density can be used to define a norm on the space of real-valued functions on $[0,1]^2$. More formally, a result of Lee and Sch\"ulke~\cite{LSch21} says that a graph $H$ is weakly norming if and only if $t(H,\cdot)$ is a convex function on the set of graphons.

Let $\HH$ be an $r$-partite $r$-graph. Recalling some definitions from~\cite{CL16} (and their vertex-folding variants from~\cite{C24}),
a sequence  $J_0, J_1, \dots$ of edge (resp.~vertex) subsets of~$\HH$ is an \emph{edge-folding sequence} (resp. \emph{vertex-folding sequence}) in $\HH$ if there is a cut involution $\phi_i$ such that  
$J_{i+1}=J_i^{+}(\phi_i)$ or $J_{i+1}=J_i^{-}(\phi_i)$ for each $i$. If a finite edge-folding sequence starts from a set consisting of a single edge and ends with $E(\HH)$, then we call it an \emph{edge-percolating sequence}.
Analogously, a \emph{vertex-percolating sequence} is a finite vertex-folding sequence that begins with a single vertex and ends with the entire part which contains the starting vertex.
To specify the part that the sequence percolates over, we say that a vertex-percolating sequence is in $A\subseteq V(\HH)$ rather than in $\HH$ if $A$ is the part that the starting vertex belongs to.

The importance of these definitions may be seen in the following theorem, which says that if an $r$-graph has an edge-percolating sequence, then it is weakly norming.

\begin{theorem}[Theorem~3.3 in~\cite{CL16}]\label{thm:reduction}
Suppose that $\HH$ is an $r$-graph which is edge-transitive under the cut involution group, the subgroup of the automorphism group of $\HH$ generated by cut involutions. If there exists an edge-percolating sequence $J_0, J_1, \dots, J_N$, then $\HH$ is weakly norming.
\end{theorem}

Therefore, to show that a graph is weakly norming, it is enough to prove that there exists an edge-percolating sequence. 
In light of~\Cref{prop:correspondence}, we can view both edge- and vertex-percolating sequences in reflection graphs in terms of elements in $W/W_I$ for an appropriate Coxeter system $(W, S)$ and $I \subseteq S$. 
A finite folding sequence that starts from one element and ends with the whole set $W/W_I$ is said to be a \textit{percolating sequence}. 
The following theorem, together with~\Cref{thm:reduction} and~\Cref{prop:correspondence}, then implies the main result of~\cite{CL16}, that every reflection graph is weakly norming.

\begin{theorem}\label{percolating in W}
    There is a percolating sequence of $W/ W_{I}$ starting from $W_I$.
\end{theorem}

By~\Cref{prop:correspondence}, \Cref{percolating in W} also proves the existence of vertex-percolating sequences.
The \emph{$(I_1,I_2\cup\dots\cup I_r;W,S)$-bigraph} is the bipartite graph between $W/W_1$ and $(W/W_2)\cup\dots\cup (W/W_r)$, where $wW_1$ and $wW_i$, $i>1$, are adjacent. In other words, we replace each $r$-edge in the $(I_1,\dots,I_r;W,S)$-graph by a star with $r-1$ leaves centred at the vertex in $W/W_1$ and simplify the resulting multigraph.
\begin{corollary}
    The $(I_1,I_2\cup\dots\cup I_r;W,S)$-bigraph has a vertex-percolating sequence in the part $W/W_1$.
\end{corollary}

This reproves the main result (Theorem 3.1) of~\cite{C24}, an interesting paper of Coregliano where he proves variations on many of the results in~\cite{CL16} and uses them to give a quantitative improvement on the result of Conlon and Lee~\cite{CL20} saying that sufficiently large blow-ups of any fixed bipartite graph satisfy Sidorenko's conjecture. While he uses slightly different terminology, for instance, saying left-cut-percolating instead of vertex-percolating, the statement above is the same as his main theorem once appropriately rephrased.

We now turn to the proof of~\Cref{percolating in W}. Following~\cite{CL16}, a \textit{stack} is a subset $U \subseteq W/W_I$ with the property that if $wW_I \in U$, then $w'W_I \in U$ for every $w'<w$.

\begin{lemma} \label{lem:stack}
Let $J\subseteq W/W_I$. If $U\subseteq J$ is a stack, then $U \subseteq J^{+}(t)$ for every $t\in T$.
\end{lemma}

\begin{proof}
Let $wW_I \in U$.
If $w \in R_t$, then $wW_I\in R_t^I\cup F_t^I$, so $t^+(wW_I)=twW_I$. As $twW_I\in U$ since $tw<w$ and $U$ is a stack, it follows that $wW_I\in J^+(t)$.
Next, if $w \in L_t$, then $wW_I\in L_t^I\cup F_t^I$, so $t^+(wW_I)=wW_I$. Then $wW_I\in J$ implies $wW_I\in J^+(t)$.   
\end{proof}

Let $U_L := \{wW_I : \ell(w^I) \leq L\}$, which is a stack by definition. The following proposition states that there exists a folding sequence starting from $U_L$ and ending with $U_{L+1}$. Applying this proposition repeatedly for each $L$ then produces a percolating sequence.

\begin{proposition}\label{thm:induction}
Let  $J \subseteq W/W_I$. If $U_L \subseteq J$, then there exists a folding sequence $J=J_0,J_1,\dots,J_n$ such that $U_{L+1}\subseteq J_n$. Furthermore, for each $i = 0,1,\dots,n-1$, $J_{i+1}=J_i^+(s_i)$ for a simple reflection $s_i\in S$.
\end{proposition}
\begin{proof}
Let $\{w_1, w_2, \dots, w_n\}$ be the set of all elements of length $L+1$ in $W^I$ and let $s_i$ be the first simple reflection in a reduced expression for $w_i$, i.e., $w_i = s_i r_i$ for some $r_i\in W^I$ of length $L$ by \Cref{thm:reduced_subwords}. In particular, $r_iW_I\in U_L$.
 
We claim that the folding sequence $J_0,J_1,\dots,J_n$ defined by $J_{i+1}=J_i^+(s_{i+1})$ ends with a set $J_n$ that contains $U_{L+1}$.
We first check that $w_iW_I\in J_i$ for $i \geq 1$.
By~\Cref{lem:stack}, $U_L\subseteq J_{i-1}$ and thus $r_iW_i\in J_{i-1}$. As $L=\ell(r_i)<\ell(w_i)$, $r_iW_I$ is in $L_{s_i}^I$, whereas $s_ir_iW_I=w_iW_I$ is in $R_{s_i}^I$. Hence, $s_i^+(w_iW_I)=r_i W_I\in J_{i-1}$, so
 $w_iW_I \in J_{i-1}^+(s_i)=J_{i}$.

It now suffices to prove that $w_iW_I$ never disappears from $J_j$ for $j>i$. Suppose that $w_iW_I\in J_j$ for some $j\geq i$. If $w_iW_I\in L_{s_j}^I \cup F_{s_j}^I$, then $s_j^+(w_iW_I)=w_iW_I\in J_j$, so $w_iW_I\in J_{j+1}=J_j^+(s_j)$.
Suppose now that $w_iW_I\in R_{s_j}^I$, i.e., $\ell((s_jw_i)^I)<\ell((w_i)^I)$. As $w_i\in W^I$ already, $(w_i)^I=w_i$ and hence $\ell((s_jw_i)^I)<L+1$. Thus, $s_jw_iW_I$ is in $U_L$, so it is also in $J_j$. Hence, $s_j^+(w_iW_I)=s_jw_iW_I\in J_j$, so $w_iW_I$ is again in $J_{j+1}$, completing the proof.
\end{proof}

This now easily implies \Cref{percolating in W}.

\begin{proof} [Proof of \Cref{percolating in W}]
Note that $U_0=\{W_I\}$. Applying \Cref{thm:induction} repeatedly from $L=0$ to the maximum length proves the existence of a percolating sequence starting from $J_0=U_0$.
\end{proof}

\subsection{Strong percolating sequences: the Janzer--Sudakov theorem}

While it is often easy to find many homomorphic copies of a given bipartite graph $H$ in another sparse $n$-vertex graph $G$, it can be challenging to show that the copies are non-degenerate, that is, that no two vertices of $H$ map to the same vertex in $G$. This issue is often the main obstacle for giving accurate estimates for the extremal number $\ext(n,H)$. For instance, this remains the chief impediment towards fully resolving the rational exponents conjecture of Erd\H{o}s and Simonovits~\cite{E81}, which states that for every rational number $r \in [1,2]$ there exists a graph $H_r$ such that $\ext(n, H_r) = \Theta(n^r)$ (though see~\cite{BC18} for a proof when $H_r$ is allowed to be a finite family and~\cite{CJ22} and its references for partial progress on the original conjecture).

An important idea used by Janzer and Sudakov in~\cite{JS24}, and observed independently by Kim, Lee, Liu and Tran in~\cite{KLLT24}, is that a suitable application of the Cauchy--Schwarz inequality can sometimes make the degree of degeneracy worse. What we mean by this is that if there are many homomorphic copies of $H$ where $v_1$ and $v_2$ both map to the same vertex in $G$, then there are also many homomorphic copies of $H$ where $v_1$, $v_2$ and another vertex $v_3$ all map to the same vertex.
Assuming $H$ has an appropriate form, we can repeat this process until we find many homomorphic copies of $H$ where all the vertices on one side of the bipartition of $H$ map to a single vertex in $G$, that is, many homomorphic copies of $H$ span a star. As is usual when studying extremal numbers, we can assume that $G$ is almost regular and, therefore, we can easily estimate the number of stars from above. Trading this off against the lower bound coming from the process above then allows us to show that the number of degenerate copies is small, which in turn implies that there must be at least one, but in fact many, non-degenerate copies of $H$ in $G$.

Janzer and Sudakov's key contribution was to show that this degeneracy-spreading process applies to graphs $H$ having vertex-percolating sequences that satisfy some additional conditions. If $H$ is a connected bipartite graph with bipartition $A\cup B$, a finite vertex-folding sequence $J_0,J_1,\dots,J_N$ of subsets of $A$ is said to be a \emph{strong vertex-percolating sequence} in $A$ if
\begin{enumerate}[(i)]
    \item $J_0$ consists of two vertices in $A$;
    \item $J_N=A$; 
    \item each $J_i$ intersects both $F_\phi\cup L_\phi$ and $F_\phi\cup R_\phi$, where $\phi$ is the cut involution such that $J_{i+1}=J_i^+(\phi)$ or $J_{i+1}=J_i^-(\phi)$.
\end{enumerate}
The graph $H$ \emph{strongly percolates} if any set $J_0$ of order two chosen from either $A$ or $B$ extends to a strong percolating sequence in the corresponding part. In their paper, Janzer and Sudakov refer to such graphs as {\it reflective graphs}, though we avoid the phrase because of its proximity to reflection graphs.

If we think of $J_i$ as being the vertices of $H$ that map to the same vertex in $G$ after $i$ steps of our degeneracy-spreading process, then, provided the third condition holds, a further application of the Cauchy--Schwarz inequality cannot eliminate the degeneracy. Moreover, by the second condition, we ultimately arrive at the entirety of $A$ or $B$, regardless of which pair $J_0$ we started with. For more details, we refer the interested reader to~\cite[Lemma~2.11]{JS24}.

We may now rephrase~\cite[Theorem~2.16]{JS24} as follows. Recall that a graph $H$ is \emph{Sidorenko} if it satisfies $t(H,G)\geq t(K_2,G)^{e(H)}$ for any graph $G$. 
The assumption that $H$ is Sidorenko is needed in order to guarantee that there are many homomorphic copies of $H$ in $G$. The key to proving the result is then in using the method we have described above to show that most of these copies must be non-degenerate. In particular, this means that we obtain not just one copy of $H$, but a \emph{supersaturation} result, meaning that there are, up to a constant factor, at least as many copies as in a random graph of the same density. 

\begin{theorem} \label{thm:strong percolation}
    Let $H$ be a connected bipartite graph with bipartition $A\cup B$ that is Sidorenko but not a tree.
    If $H$ strongly percolates, then there are constants $C, c >0$ such that any $n$-vertex graph $G$ with edge density $p\geq Cn^{-\frac{v-t-1}{e-t}}$, where $v=v(H)$, $e=e(H)$ and $t=\max\{|A|,|B|\}$, contains at least $cn^vp^e$ copies of $H$.
\end{theorem}

We wish to study which reflection graphs satisfy the conditions of this theorem. To this end, let $H$ be the $(I,J;W,S)$-graph. By Theorems~\ref{thm:reduction} and \ref{percolating in W}, this is a weakly norming graph. In particular, $H$ is Sidorenko, as was noted by Hatami~\cite{H10}. 
We now give conditions under which a reflection graph $H$ contains a cycle or is connected. In particular, this result implies that if $I$ and $J$ are proper subsets of $S$ with $I\cup J=S$, then the $(I,J;W,S)$-graph is connected with a cycle.

\begin{lemma}\label{lem:notree}
    Let $H$ be the $(I,J;W,S)$-graph. Then 
    \begin{enumerate}[(i)]
        \item $H$ contains a cycle unless $I\subseteq J$ or $J\subseteq I$ and
        \item $H$ is connected if $I\cup J=S$.
    \end{enumerate}
\end{lemma}

\begin{proof}
(i) Suppose that both $I\setminus J$ and $J\setminus I$ are nonempty and
    let $s\in I\setminus J$ and $s'\in J\setminus I$. Then there exists a minimum $m=m(s,s')$ such that $(ss')^m = 1$. In particular, $W_I, sW_J, ss'W_I,\dots, (ss')^mW_I$ forms a cycle of length $2m$.
    
(ii) Suppose that $I\cup J=S$ and let $w\in W$ be $w=s_1s_2\cdots s_\ell$ for $s_i\in S$. Then each $s_i$ is in $I$ or $J$, so $w$ can be written as $w=w_1w_2\cdots w_{2k}$, where each $w_{2i-1}\in W_I$ and $w_{2i}\in W_J$.
    Note that $W_I$ and $w_1W_J$ are adjacent as $w_1\in W_I\cap w_1W_J$. Similarly, $w_1W_J$ and $w_1w_2W_I$ are adjacent as $w_1w_2\in w_1w_2W_I\cap w_1W_J$. Again, $w_1w_2W_I$ and $w_1w_2w_3 W_J$ are adjacent since $w_1w_2w_3\in w_1w_2W_I\cap w_1w_2w_3W_J$. 
    Repeating this process gives a walk from $W_I$ to $wW_I$. 
    The same argument with $I$ and $J$ swapped gives a walk from $W_J$ to $wW_J$, so $H$ is connected.
\end{proof}

As the the $(I,J;W,S)$-graph is vertex-transitive on each side, we may always assume that $J_0$ consists of $\{W_I,wW_I\}$ for some $w\in W$ such that $wW_I\neq W_I$ or $\{W_J,wW_J\}$ for some $w\in W$ such that $wW_J\neq W_J$.
We formalise this reduction as follows.

\begin{lemma}\label{lem:reduction}
    Let $H$ be the $(I,J;W,S)$-graph for $I,J\subsetneq S$ with $I\cup J=S$. Then $H$ strongly percolates if and only if there exists a strong percolating sequence starting with $\{W_I,wW_I\}$ for each $w \in W$ such that $wW_I\neq W_I$ and starting with $\{W_J,wW_J\}$ for each $w \in W$ such that $wW_J\neq W_J$.
\end{lemma}

\begin{proof}
    Suppose that there exists a strong percolating sequence $J_0,J_1,\dots,J_N$ with $J_0=\{W_I,wW_I\}$ for each $w\in W$.
    Let $J_0'=\{w_1W_I, w_2W_I\}$ with $w_1W_I\neq w_2W_I$. Let $\psi$ be an automorphism that sends $W_I$ to $w_1W_I$. 
    Such an automorphism exists since each simple reflection $s\in S$ gives a cut involution of $H$ by~\Cref{thm:cut_involution} and so, if $w_1=s_1s_2\cdots s_k$ for $s_i\in S$, then the composition of the corresponding cut involutions gives an automorphism $\psi$ that sends $W_I$ to $w_1W_I$. Moreover, $\psi$ sends $w_1^{-1}w_2W_I$ to $w_2W_I$. Let $J_0,J_1,\dots, J_N$ be a strong percolating sequence starting from $J_0=\{W_I,w_1^{-1}w_2W_I\}$.
    
    For any cut involution $\phi$ of $H$, $\psi\phi\psi^{-1}$ is again a cut involution that sends $\psi(v)$ to $\psi(\phi(v))$ with $L=\psi(L_\phi)$, $R=\psi(R_\phi)$ and $F=\psi(F_\phi)$.
    Thus, taking $J_{i+1}':=J_i'^+(\psi\phi_i\psi^{-1})$, where $\phi_i$ is the cut involution such that $J_{i+1}=J_i^+(\phi_i)$, gives a strong percolating sequence starting from $J_0'$. 
\end{proof}

We now introduce an algebraic analogue of strong percolating sequences in a Coxeter system $(W,S)$. A folding sequence $J_0, J_1, \dots , J_N\subseteq W/W_I$ is a \emph{strong folding sequence} if each $J_i$ has nonempty intersection with both $L^I_{t}\cup F^I_t$ and $R^I_t\cup F^I_t$, where $t$ is the reflection such that $J_{i+1}=J_i^+(t)$ or $J_{i+1}=J_i^-(t)$.
A \emph{strong percolating sequence} is a strong folding sequence that starts from a pair of the form $\{W_I,wW_I\}$ and ends with the whole collection of left cosets $W/W_{I}$.
We say that $W/W_I$ \emph{strongly percolates} if, for any $w\in W \setminus W_I$ (or, equivalently, any $w\in W^I \setminus \{e\}$), there exists a strong percolating sequence starting from $\{W_I,wW_I\}$.
The importance of these definitions is that, by  \Cref{lem:notree,lem:reduction}, we have the following variant of \Cref{thm:strong percolation} for reflection graphs.

\begin{proposition}
    Let $H$ be the $(I,J;W,S)$-graph for $I,J\subsetneq S$ with $I\cup J=S$. If both $W/W_I$ and $W/W_J$ strongly percolate, then there are constants $C, c > 0$ such that any $n$-vertex graph $G$ with edge density $p\geq Cn^{-\frac{v-t-1}{e-t}}$, where $v=v(H)$, $e=e(H)$ and $t=\max\{|W/W_I|,|W/W_J|\}$,
    contains at least $cn^vp^e$ copies of $H$.
\end{proposition}

To apply this result, we would like to determine those proper subsets $I$ of $S$ for which $W/W_I$ strongly percolates. We answer this question completely.

\begin{theorem} \label{main}
    Let $I\subsetneq S$. Then $W/W_I$ strongly percolates if and only if $|I|=|S|-1$.
\end{theorem}

\begin{proof}
Suppose first that $|I|=|S|-1$ with $s \in S \setminus I$. By \Cref{lem:reduction}, it suffices to find a strong percolating sequence starting with $J_0 = \{W_I, wW_I\}$ for each $w \in W^I\setminus \{e\}$. Let $w = s_1 s_2 \cdots s_\ell$ be a reduced expression for $w$. Then $s_\ell$ must be $s$, as otherwise it contradicts the fact that $W^I$ consists of those $w\in W$ of minimum length in the coset $wW_I$. Let $k$ be the smallest index $i$ such that $s_i = s$.

The strong percolating sequence $J_0, J_1, \dots, J_{N}$ which we construct has two parts. The first part, up to $J_{k}$, is chosen so as to guarantee that $J_{k}$ contains $U_1= \{W_I, s W_I\}$, a stack of length one in $W/W_I$. We may then assume that $J_{k}=U_1$ and show that the second half, $J_{k+1},J_{k+2},\dots,J_N$, satisfies $J_N=W/W_I$.

For each $0\leq i<k$, we make a choice between  $J_{i+1} = J_{i}^+(s_{i+1})$ or $J_{i}^-(s_{i+1})$ that maintains the condition that each $J_i$ contains $\{W_I, s_{i+1} s_{i+2} \cdots s_\ell W_I\}$ as a subset and intersects both $F^I_{s_{i+1}}\cup L^I_{s_{i+1}}$ and $F^I_{s_{i+1}} \cup R^I_{s_{i+1}}$. 
To do this, for $0 \leq i \leq k-2$, we take $J_{i+1} = J_{i}^-(s_{i+1})$. In this case, $W_I \in F^I_{s_{i+1}}$, while $s_{i+1} s_{i+2} \cdots s_\ell W_I \in R^I_{s_{i+1}}$, so $J_{i+1}$ contains $\{W_I, s_{i+2}\cdots s_\ell W_I\}$. For the final step, take $J_k = J_{k-1}^+(s_{k})= J_{k-1}^+(s)$. But $W_I \in L^I_{s}$ generates $\{W_I, s W_I\}$ in $J_k$, as desired.

For the second half of the sequence, starting with $J_k$, which contains $U_1 = \{W_I, sW_I\}$, and iterating \Cref{thm:induction} for $L=1,2,\dots$ eventually gives a percolating sequence that ends with $J_N=W/W_I$ and satisfies $J_{i+1}=J_i^+(s_i')$ or $J_{i+1}=J_i^-(s_i')$ for some $s_i'\in S$. We claim that this is a strong percolating sequence. Indeed, by \Cref{lem:stack}, $U_1= \{W_I, sW_I\}$ is contained in every $J_i$ for $i\geq k$. 
If $s_i'=s$, then $W_I \in J_i \cap L^I_{s}$, while $sW_I \in J_i \cap R^I_{s}$. Otherwise, if $s_i' \in I$, then $W_I \in J_i \cap F^I_{s_i'}$. 

Conversely, suppose $|I| \leq |S|-2$ with two distinct simple reflections $s$ and $s'$ that are not in $I$. We shall prove that there is no strong percolating sequence starting from $J_0 = \{W_I, s W_I\}$. Suppose, for the sake of contradiction, that there is a strong percolating sequence $J_0, J_1, \dots, J_N = W/W_I$.

Among all the reflections used to construct the strong percolating sequence, let $t \in T$ be the first reflection that contains $s'$ in its reduced expressions. Suppose that $t$ is used at the $(k+1)$-st step, i.e.,  $J_{k+1}=J^+_k(t)$ or $J_{k+1}=J^-_k(t)$. We claim that the whole of $J_k$ is contained in the left side $L^I_{t}$, which means that the sequence is not a strong percolating sequence.

Let $w W_I \in J_k$, noting that $w \in W_{I \cup \{s\}}$, and let $w=s_1s_2\cdots s_\ell$ be a reduced expression for $w$. 
If $\ell(tw) < \ell(w)$, then, by \Cref{thm:strongexchange}, there is $i\in [\ell]$ such that $t=us_iu^{-1}$ for $u=s_1s_2\cdots s_{i-1}$. 
However, this contradicts \Cref{cor:generating}, since the simple reflection $s'$ contained in any reduced expression for $t$ is not in that for $us_iu^{-1}$. Thus, we must have $\ell(tw) \geq \ell(w)$. Moreover, by applying the order-preserving projection $P^I$, $\ell((tw)^I) \geq \ell(w^I)$.

It remains to show that $\ell((tw)^I) = \ell(w^I)$ is also impossible. By \Cref{fixed points}, this equality is equivalent to $tw W_I = wW_I$, that is, $t=ww'w^{-1}$ for some $w'\in W_I$. However, this again contradicts \Cref{cor:generating}. Hence, since $\ell((tw)^I) > \ell(w^I)$, we have $wW_I \in L^I_{t}$ by the definition of $L^I_{t}$, as required.
\end{proof}

We therefore see that generalised face-incidence graphs are strongly percolating, completing the proof of~\Cref{thm:main}. It may even be that this is a complete classification of those reflection graphs that strongly percolate, i.e., that the $(I, J; W, S)$-graph strongly percolates if and only if $|I| = |J| = |S| - 1$. Unfortunately, this is not implied by~\Cref{main}, since the converse to~\Cref{thm:cut_involution} does not hold, that is, cut involutions in reflection graphs do not necessarily correspond to reflections in the associated Coxeter group (see, for example,~\cite[Example 4.10]{CL16}). 
To say more, let us look more closely, but from our point of view, at the  $2$-blow-up of $C_6$, which was already observed to be non-strongly-percolating in~\cite{JS24}.

\begin{example}
Let $G_1$ and $G_2$ be bipartite graphs with bipartitions $A_1 \cup B_1$ and $A_2\cup B_2$, respectively. The \emph{bipartite tensor product} of $G_1$ and $G_2$, denoted by $G_1 \times G_2$, is the bipartite graph with bipartition $(A_1 \times A_2)\cup (B_1 \times B_2)$, where two vertices $(a_1, a_2)$ and $(b_1, b_2)$ are adjacent if and only if $a_i$ and $b_i$ are adjacent in $G_i$ for both $i = 1$ and $2$. As noted in \cite{CL16}, the bipartite tensor product of two reflection graphs is again a reflection graph. But the $2$-blow-up of $C_6$ can be expressed as the bipartite tensor product of $C_6$ and $K_{2,2}$, both of which are reflection graphs, so it is also a reflection graph.

\begin{figure}[h]
\centering
    \begin{tikzpicture}[scale=1]
        \foreach \i in {0,...,5}
			\node[vertex] at (360/6*\i : 0.7) (v\i) {};
		\foreach \i in {0,...,5}
			\node[vertex] at (360/6*\i : 1.3) (u\i) {};
        \foreach \i [evaluate={\j=int(mod(\i+1,6));}] in {0,...,5}
            {\draw (u\i)--(u\j)--(v\i);
            \draw (u\i)--(v\j)--(v\i);}
		\end{tikzpicture}
        \caption{The $2$-blow-up of $C_6$}
\end{figure}
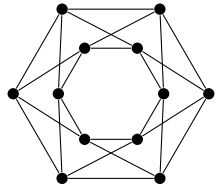

More explicitly, using that $C_6$ is the incidence graph between the vertices and edges of an equilateral triangle, we may view $C_6$ as the $(\{s_1\}, \{s_2\}; S_3, \{s_1, s_2\})$-graph, where $S_3$ is the permutation group on three vertices, $s_1 = (1, 2)$ and $s_2 = (2, 3)$. Taking $s_1'$ and $s_2'$ to be two commuting involutions, we can view $K_{2,2}$ as the $(\{s_1'\}, \{s_2'\}; W', \{s_1', s_2'\})$-graph, where $W'= \langle s_1',s_2'\vert (s_1')^2=(s_2')^2=1, s_1's_2'=s_2's_1'\rangle $ is the Coxeter group generated by $\{s_1', s_2'\}$.

Their bipartite tensor product, which is the $2$-blow-up of $C_6$, may then be viewed as the $(\{s_1, s_1'\}, \{s_2, s_2'\}; W, \{s_1, s_2, s_1', s_2'\})$-graph, where $W=S_3\times W'$ is the Coxeter group which is the direct product of $S_3$ and $W'$. Not all cut involutions in this graph correspond to reflections. However, if we restrict ourselves so that we can only apply reflections, rather than any cut involution, when creating a folding sequence, then we cannot strongly percolate. Indeed, since the $2$-blow-up of $C_6$ is an $(I,J;W,S)$-graph with $|S| = 4$ and $|I| = |J| 
= 2$, \Cref{main} implies that it does not strongly percolate through reflections. Of course, it is easily checked by hand that it also does not strongly percolate without this restriction. 

We think that this should be a more general phenomenon, that is, that if a reflection graph is not strongly percolating through reflections, then it is not strongly percolating. Rephrasing this in the contrapositive, we believe that if a graph strongly percolates, then it also strongly percolates through reflections alone. If this were true, and we leave it as an open problem to show that it is, then, together with~\Cref{main}, it would yield the classification of strongly percolating reflection graphs suggested above, namely, that the $(I, J; W, S)$-graph strongly percolates if and only if $|I| = |J| = |S| - 1$.
\end{example}

On the other hand, we now give several more examples of graphs $H$ for which~\Cref{thm:main} gives a new upper bound for $\ext(n,H)$, confirming~\Cref{conj:CL} for these graphs. 
\Cref{thm:main} is particularly powerful when the bipartition of $H$ is \emph{balanced} in the sense that the two parts in the bipartition are of the same size. In this case, the reflection graph $H$ becomes $r$-regular for some $r$.

\begin{corollary}\label{cor:balanced}
    Let $H$ be an $r$-regular generalised face-incidence graph. Then $\ext(n,H)=O(n^{2-c})$ where $c=\frac{v(H)-2}{v(H)(r-1)}$. In particular, $H$ satisfies~\Cref{conj:CL} unless $H=K_{r,r}$.
\end{corollary}

To see why the `in particular' part works, it is enough to observe that 
$$\frac{v(H)-2}{v(H)(r-1)} - \frac{1}{r} = \frac{v(H)-2r}{v(H)(r-1)r} \ge 0,$$
with equality if and only if $v(H) = 2r$, in which case $H$ must be $K_{r,r}$. Note that this already recovers the Janzer--Sudakov result that~\Cref{conj:CL} holds for hypercubes and bipartite Kneser graphs.

As an application of~\Cref{cor:balanced}, we can give a non-trivial upper bound for the $(0,3)$-incidence graph $H$ of the 24-cell, the 4-dimensional self-dual regular polytope associated with the exceptional finite reflection group $F_4$. 
The 24-cell has 24 vertices and 24 octahedral faces, with each vertex contained in six faces. The 24-cell is shown in~\Cref{subfig:24cell}, with three of its 24 octahedral faces highlighted. 
Its $(0,3)$-incidence graph $H$ is the $6$-regular bipartite graph on $48$ vertices given in~\Cref{subfig:incidence}. For this graph $H$,~\Cref{cor:balanced} gives $\ext(n,H)=O(n^{109/60})$, beating F\"uredi's bound of $O(n^{11/6})$ and thereby confirming~\Cref{conj:CL} in this case.

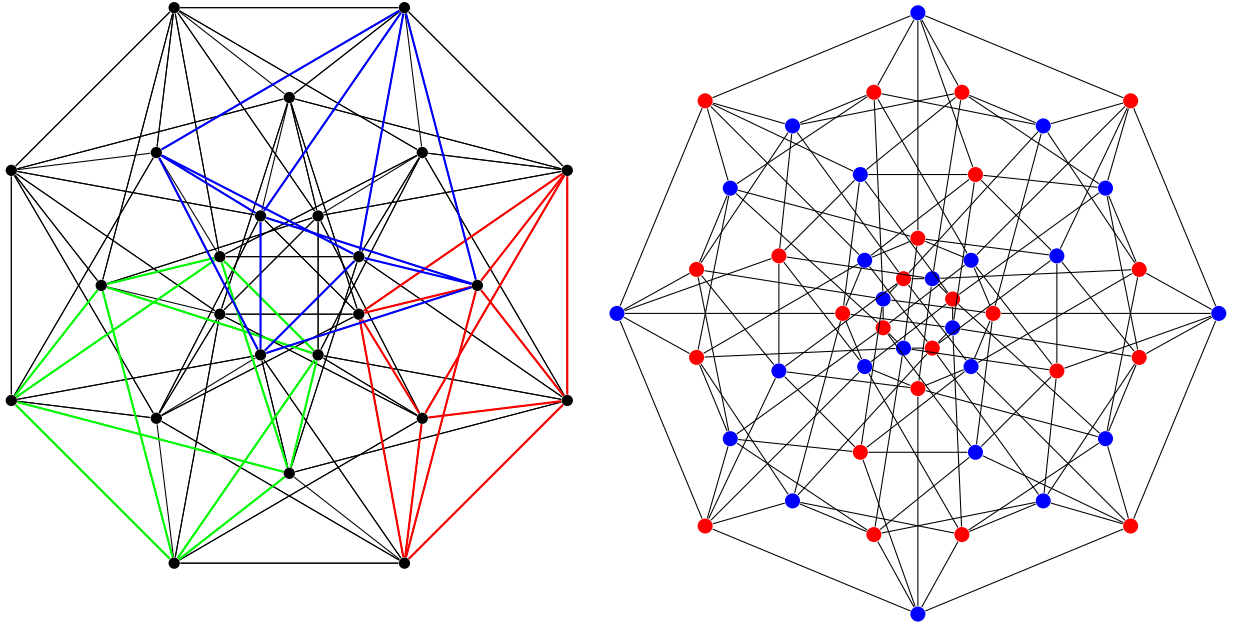
\begin{figure}[!thb]\label{fig:24-cell}

\begin{minipage}{0.50\textwidth}
\centering
\subfloat[\small The 24-cell \label{subfig:24cell}]{

\vspace{7.5mm}
\begin{tikzpicture}
    \foreach \i in {0,...,7}{
    \node[vertex] at (360/8*\i+22.5:1) (a\i) {};
    \node[vertex] at (360/8*\i:2.5) (b\i) {};
    \node[vertex] at (360/8*\i+22.5:4) (c\i) {};
    }

    \foreach \i
    [evaluate=
    {\a=int(mod(\i+1,8)); \b=int(mod(\i+2,8)); \c=int(mod(\i+3,8)); \d=int(mod(\i+4,8)); \e=int(mod(\i+5,8)); \f=int(mod(\i+6,8)); \g=int(mod(\i+7,8));}]
    in {0,...,7} {
    \draw (a\i)--(a\c);
    \draw (a\i)--(a\e);
    \draw (a\i)--(b\i);
    \draw (a\i)--(b\a);
    \draw (a\i)--(b\c);
    \draw (a\i)--(b\f);
    \draw (a\i)--(c\a);
    \draw (a\i)--(c\g);

    \draw (b\i)--(a\b);
    \draw (b\i)--(a\e);
    \draw (b\i)--(a\g);
    \draw (b\i)--(c\i);
    \draw (b\i)--(c\a);
    \draw (b\i)--(c\f);
    \draw (b\i)--(c\g);
    
    \draw (c\i)--(a\a);
    \draw (c\i)--(a\g);
    \draw (c\i)--(b\a);
    \draw (c\i)--(b\b);
    \draw (c\i)--(b\g);
    \draw (c\i)--(c\a);
    \draw (c\i)--(c\g);
    }

    \draw[red,thick] (a7)--(b7)--(c7)--(b0)--(a7);
    \draw[red,thick] (c0)--(a7)--(c6)--(b0)--(c0);
    \draw[red,thick] (c0)--(c7)--(c6)--(b7)--(c0);

    \draw[green,thick] (a3)--(c4)--(c5)--(a6)--(a3);
    \draw[green,thick] (b4)--(c4)--(b6)--(c5)--(b4);
    \draw[green,thick] (b4)--(a3)--(b6)--(a6)--(b4);

    \draw[blue,thick] (a2)--(b3)--(a0)--(b0)--(a2);
    \draw[blue,thick] (c1)--(a2)--(a5)--(b3)--(c1);
    \draw[blue,thick] (c1)--(b0)--(a5)--(a0)--(c1);

\end{tikzpicture}
}

\end{minipage}%
\begin{minipage}{0.50\textwidth}

\centering
\subfloat[\small The $(0,3)$-incidence graph of the $24$-cell \label{subfig:incidence}]{

\begin{tikzpicture}
    \foreach \i in {0,1,2,3}{
    \node[vertexred] at (90*\i+22.5:0.5) (x\i) {};
    \node[vertexblue] at (90*\i+67.5:0.5) (a\i) {};

    \node[vertexred] at (90*\i:1) (y\i) {};
    \node[vertexblue] at (90*\i+45:1) (b\i) {};

    \node[vertexblue] at (90*\i+22.5:2) (c\i) {};
    \node[vertexred] at (90*\i+67.5:2) (z\i) {};

    \node[vertexblue] at (90*\i:4) (e\i) {};
    \node[vertexred] at (90*\i+45:4) (s\i) {};
    }

    \foreach \i in {0,2,4,6}{
    \node[vertexred] at (45*\i-11.25:3) (w\i) {};
    \node[vertexblue] at (45*\i+33.75:3) (d\i) {};
    }

    \foreach \i in {1,3,5,7}{
    \node[vertexred] at (45*\i-33.75:3) (w\i) {};
    \node[vertexblue] at (45*\i+11.25:3) (d\i) {};
    }

    \foreach \i
    [evaluate=
    {\a=int(mod(\i,4)); \b=int(mod(\i+1,4)); \c=int(mod(\i+2,4)); \d=int(mod(\i+3,4)); 
    \aa=int(mod(2*\i,8)); \bb=int(mod(2*\i+1,8)); \cc=int(mod(2*\i+2,8)); \dd=int(mod(2*\i+3,8)); \ee=int(mod(2*\i+4,8)); \ff=int(mod(2*\i+5,8)); \gg=int(mod(2*\i+6,8)); \hh=int(mod(2*\i+7,8));}]
    in {0,1,2,3} {
    \draw (x\i)--(a\a);
    \draw (x\i)--(a\d);
    \draw (x\i)--(b\c);
    \draw (x\i)--(c\b);
    \draw (x\i)--(d\hh);
    \draw (x\i)--(d\aa);

    \draw (y\i)--(a\b);
    \draw (y\i)--(b\a);
    \draw (y\i)--(b\d);
    \draw (y\i)--(c\d);
    \draw (y\i)--(d\bb);
    \draw (y\i)--(e\a);

    \draw (z\i)--(a\d);
    \draw (z\i)--(b\b);
    \draw (z\i)--(c\a);
    \draw (z\i)--(c\b);
    \draw (z\i)--(d\aa);
    \draw (z\i)--(e\b);

    \draw (s\i)--(b\a);
    \draw (s\i)--(c\a);
    \draw (s\i)--(d\aa);
    \draw (s\i)--(d\bb);
    \draw (s\i)--(e\a);
    \draw (s\i)--(e\b);

    \draw (w\aa)--(a\d);
    \draw (w\aa)--(c\a);
    \draw (w\aa)--(d\aa);
    \draw (w\aa)--(d\gg);
    \draw (w\aa)--(d\hh);
    \draw (w\aa)--(e\a);

    \draw (w\bb)--(a\a);
    \draw (w\bb)--(b\d);
    \draw (w\bb)--(d\aa);
    \draw (w\bb)--(d\bb);
    \draw (w\bb)--(d\hh);
    \draw (w\bb)--(e\a);
    }

\end{tikzpicture}

}
\end{minipage}%
\caption{The 24-cell and its $(0,3)$-incidence graph}
\end{figure}

For the rest of this section, we will look for non-regular examples, focusing on reflection graphs arising from the symmetric group $S_m$. 
A standard way to describe $S_m$ as a Coxeter system is to use the set of adjacent transpositions $s_i=(i,i+1)$ for $i = 1, \dots, m-1$ as the set $S$ of simple reflections. Note that we may write any permutation $x \in S_m$ as $x = x_1 x_2 \cdots x_m$, where $x_i=x(i)$. Assuming this notation, for $k \in [m-1]$, let $S_{m}^{(k)}:=\{x \in S_m: x_1 < \cdots < x_k \text{ and } x_{k+1}<\cdots <x_m\}$. We have the following straightforward lemma.

\begin{lemma}[Lemma $2.4.7$ in~\cite{BB05}]\label{lem:(S_n)_I}
    Let $I = S \setminus \{s_k\}$. Then $(S_m)_I \cong S_k \times S_{m-k}$ and $(S_m)^I = S_m ^{(k)}$. 
    In particular, $(S_m)^I$ can be identified with $\binom{[m]}{k}$, the set of all $k$-subsets of $[m]$.
\end{lemma}

The \emph{(m,a,b)-inclusion graph} for $1\leq a<b <m$ is the bipartite graph between the set of $a$-element subsets and the set of $b$-element subsets of $[m]$, where two subsets $U \in \binom{[m]}{a}$ and $V \in \binom{[m]}{b}$ are adjacent if and only if $U \subseteq V$. 
\Cref{lem:(S_n)_I} implies that the $(m,a,b)$-inclusion graph is isomorphic to the $(S\setminus \{s_{a}\}, S\setminus \{s_{b}\}; S_m, S)$-graph. 
Consequently, any such graph strongly percolates and so \Cref{thm:incidence_main} gives an upper bound for its extremal number. We now analyse when this bound is stronger than F\"{u}redi's classical bound.

\begin{example}\label{ex:S_n}
Let $H$ be the $(m,a,b)$-inclusion graph with bipartition $X\cup Y$, where $X = \binom{[m]}{a}$ and $Y = \binom{[m]}{b}$. Then $e(H)=\binom{m}{a}\binom{m-a}{b-a}=\binom{m}{b}\binom{b}{a}$.

Suppose first that $a<b\leq m/2$. Then $|X|\leq |Y|$, so the F\"uredi bound gives $\ext(n,H)\leq n^{2-c}$ with $c=1/\binom{b}{a}$. On the other hand,~\Cref{thm:main} gives $\ext(n,H)\leq n^{2-c'}$ with
\begin{align*}
    c'= \frac{{\binom{m}{a}} -1}{\binom{m}{b} \binom{b}{a} - \binom{m}{b}}.
\end{align*}
Therefore, $c'>c$ if and only if 
\begin{align*}
    1-\dfrac{1}{\binom{b}{a}} < \dfrac{\binom{m}{a}-1}{\binom{m}{b}}.
\end{align*}
The range of $(a,b)$ for which this inequality is satisfied is quite narrow. For example, if $b=O_m(1)$, the inequality is never satisfied for large enough $m$. Even for $b=\lfloor m/2\rfloor$, $a=b-1$ is the only value for which the inequality is satisfied and even then only for $m$ even. 
The case $m/2<a<b$ also follows from the same analysis, as one can take $m-a$ and $m-b$ instead of $b$ and $a$, respectively, by passing to the complement of each set.

Next, suppose that $a\leq m/2<b$. We may assume that $|a-m/2| \leq |b-m/2|$, as otherwise one can again use $m-a$ and $m-b$. A similar analysis to that above gives that~\Cref{thm:main} improves on F\"uredi's bound if the inequality 
\begin{align}\label{case2}
    1-\dfrac{1}{\binom{m-a}{b-a}} < \dfrac{\binom{m}{b}-1}{\binom{m}{a}}.
\end{align}
holds. This again has very few solutions, though it is satisfied when $b=a+1$ for $a=\lfloor m/2\rfloor$, which, for $m$ even, yields a graph which is isomorphic to the graph with $b = \lfloor m/2\rfloor$ and $a = b -1$. 
It is also satisfied when $b=m-a$, again recovering the case of bipartite Kneser graphs.
\end{example}

\section{Concluding remarks}

As we have seen in~\Cref{cor:balanced} and~\Cref{ex:S_n}, \Cref{thm:main} is most efficient for bipartite graphs with an almost balanced bipartition. This is due to the fact that one cannot generally control the part in which the degeneracy happens. An exception is the simple proof of the classical K\H{o}vari--S\'os--Tur\'an theorem, saying that $\ext(n,K_{r,s})=O(n^{2-1/r})$ for $r\leq s$, which proceeds by counting the number of non-degenerate $r$-leaf stars, avoiding degeneracy on the $r$-vertex side of $K_{r,s}$ from the beginning. If there were an argument that allowed us to count the number of copies of some unbalanced incidence graph $H$ while avoiding degeneracy on the smaller side of the bipartition, then it would allow us to take $t$ in~\Cref{thm:main} to be the size of the smaller side, potentially giving a better bound for $\ext(n,H)$. However, following through on this plan seems to require additional ideas.

It is also worth remarking that much of what we have said about strong percolation with respect to reflection graphs carries over to reflection hypergraphs. If we define strong percolation for $r$-graphs in the natural way, we see, from~\Cref{main}, that the $(I_1,\dots,I_r;W,S)$-graph strongly percolates if $|I_i| = |S| - 1$ for all $1 \le i \le r$. That is, generalised face-incidence hypergraphs strongly percolate. Unfortunately, we do not have an obvious analogue of~\Cref{thm:strong percolation} to which we can apply this result. However, it may still be interesting to explore this direction further. 

\bibliographystyle{plainurl}
\bibliography{references}

\begin{thebibliography}{10}

\bibitem{AKS03}
Noga Alon, Michael Krivelevich, and Benny Sudakov.
\newblock Tur\'an numbers of bipartite graphs and related {R}amsey-type questions.
\newblock {\em Combin. Probab. Comput.}, 12(5-6):477--494, 2003.
\newblock \href {https://doi.org/10.1017/S0963548303005741} {\path{doi:10.1017/S0963548303005741}}.

\bibitem{BT11}
Rahil Baber and John Talbot.
\newblock Hypergraphs do jump.
\newblock {\em Combin. Probab. Comput.}, 20(2):161--171, 2011.
\newblock \href {https://doi.org/10.1017/S0963548310000222} {\path{doi:10.1017/S0963548310000222}}.

\bibitem{BB05}
Anders Bj{\"o}rner and Francesco Brenti.
\newblock {\em Combinatorics of Coxeter Groups}.
\newblock Graduate Texts in Mathematics. Springer Berlin Heidelberg, 2005.
\newblock URL: \url{https://books.google.co.uk/books?id=OsaeruFjCRsC}.

\bibitem{Bu24}
Boris Bukh.
\newblock Extremal graphs without exponentially small bicliques.
\newblock {\em Duke Math. J.}, 173(11):2039--2062, 2024.
\newblock \href {https://doi.org/10.1215/00127094-2023-0043} {\path{doi:10.1215/00127094-2023-0043}}.

\bibitem{BC18}
Boris Bukh and David Conlon.
\newblock Rational exponents in extremal graph theory.
\newblock {\em J. Eur. Math. Soc. (JEMS)}, 20(7):1747--1757, 2018.
\newblock \href {https://doi.org/10.4171/JEMS/798} {\path{doi:10.4171/JEMS/798}}.

\bibitem{CJ22}
David Conlon and Oliver Janzer.
\newblock Rational exponents near two.
\newblock {\em Adv. Comb.}, 2022:Paper No. 9, 10 pp., 2022.
\newblock \href {https://doi.org/10.19086/aic.2022.9} {\path{doi:10.19086/aic.2022.9}}.

\bibitem{CJL21}
David Conlon, Oliver Janzer, and Joonkyung Lee.
\newblock More on the extremal number of subdivisions.
\newblock {\em Combinatorica}, 41(4):465--494, 2021.
\newblock \href {https://doi.org/10.1007/s00493-020-4202-1} {\path{doi:10.1007/s00493-020-4202-1}}.

\bibitem{CL16}
David Conlon and Joonkyung Lee.
\newblock Finite reflection groups and graph norms.
\newblock {\em Adv. Math.}, 315:130--165, 2017.
\newblock \href {https://doi.org/10.1016/j.aim.2017.05.009} {\path{doi:10.1016/j.aim.2017.05.009}}.

\bibitem{CL21}
David Conlon and Joonkyung Lee.
\newblock On the extremal number of subdivisions.
\newblock {\em Int. Math. Res. Not.}, 2021(12):9122--9145, 2021.
\newblock \href {https://doi.org/10.1093/imrn/rnz088} {\path{doi:10.1093/imrn/rnz088}}.

\bibitem{CL20}
David Conlon and Joonkyung Lee.
\newblock Sidorenko's conjecture for blow-ups.
\newblock {\em Discrete Anal.}, 2:1--14, 2021.
\newblock \href {https://doi.org/10.19086/da} {\path{doi:10.19086/da}}.

\bibitem{C24}
Leonardo~N. Coregliano.
\newblock Left-cut-percolation and induced-{S}idorenko bigraphs.
\newblock {\em SIAM J. Disc. Math.}, 38(2):1586--1629, 2024.
\newblock \href {https://doi.org/10.1137/22M1526794} {\path{doi:10.1137/22M1526794}}.

\bibitem{C34}
Harold Scott~MacDonald Coxeter.
\newblock Discrete groups generated by reflections.
\newblock {\em Ann. of Math.}, 35(3):588--621, 1934.
\newblock \href {https://doi.org/10.2307/1968753} {\path{doi:10.2307/1968753}}.

\bibitem{E81}
Paul. Erd\H{o}s.
\newblock On the combinatorial problems which {I} would most like to see solved.
\newblock {\em Combinatorica}, 1(1):25--42, 1981.
\newblock \href {https://doi.org/10.1007/BF02579174} {\path{doi:10.1007/BF02579174}}.

\bibitem{Fu91}
Zolt{\'a}n F{\"u}redi.
\newblock On a {T}ur{\'a}n type problem of {E}rd{\H{o}}s.
\newblock {\em Combinatorica}, 11(1):75--79, 1991.
\newblock \href {https://doi.org/10.1007/BF01375476} {\path{doi:10.1007/BF01375476}}.

\bibitem{H10}
Hamed Hatami.
\newblock Graph norms and {S}idorenko's conjecture.
\newblock {\em Israel J. Math.}, 175(1):125--150, 2010.
\newblock \href {https://doi.org/10.1007/s11856-010-0005-1} {\path{doi:10.1007/s11856-010-0005-1}}.

\bibitem{HHKNR12}
Hamed Hatami, Jan Hladk{\'y}, Daniel Kr\'{a}l', Serguei Norine, and Alexander~A. Razborov.
\newblock Non-three-colourable common graphs exist.
\newblock {\em Combin. Probab. Comput.}, 21(5):734--742, 2012.
\newblock \href {https://doi.org/10.1017/S0963548312000107} {\path{doi:10.1017/S0963548312000107}}.

\bibitem{H92}
James~E. Humphreys.
\newblock {\em Reflection Groups and Coxeter Groups}.
\newblock Cambridge Studies in Advanced Mathematics. Cambridge University Press, 1992.
\newblock URL: \url{https://books.google.co.in/books?id=ODfjmOeNLMUC}.

\bibitem{Ja19}
Oliver Janzer.
\newblock Improved bounds for the extremal number of subdivisions.
\newblock {\em Electron. J. Combin.}, 26(3):Paper No. 3.3, 6 pp., 2019.
\newblock \href {https://doi.org/10.37236/8262} {\path{doi:10.37236/8262}}.

\bibitem{JS24}
Oliver Janzer and Benny Sudakov.
\newblock On the {T}ur{\'a}n number of the hypercube.
\newblock {\em Forum. Math. Sigma}, 12:e38, 19 pp., 2024.
\newblock \href {https://doi.org/10.1017/fms.2024.27} {\path{doi:10.1017/fms.2024.27}}.

\bibitem{KLLT24}
Jaehoon Kim, Joonkyung Lee, Hong Liu, and Tuan Tran.
\newblock Rainbow cycles in properly edge-colored graphs.
\newblock {\em Combinatorica}, 44(4):909--919, 2024.
\newblock \href {https://doi.org/10.1007/s00493-024-00101-7} {\path{doi:10.1007/s00493-024-00101-7}}.

\bibitem{KRSz96}
J\'anos Koll\'ar, Lajos R\'onyai, and Tibor Szab\'o.
\newblock Norm-graphs and bipartite {T}ur\'an numbers.
\newblock {\em Combinatorica}, 16(3):399--406, 1996.
\newblock \href {https://doi.org/10.1007/BF01261323} {\path{doi:10.1007/BF01261323}}.

\bibitem{LSch21}
Joonkyung Lee and Bjarne Sch{\"u}lke.
\newblock Convex graphon parameters and graph norms.
\newblock {\em Israel J. Math.}, 242(2):549--563, 2021.
\newblock \href {https://doi.org/10.1007/s11856-021-2112-6} {\path{doi:10.1007/s11856-021-2112-6}}.

\bibitem{L08}
L{\'a}szl{\'o} Lov\'{a}sz.
\newblock Graph homomorphisms: open problems.
\newblock Unpublished manuscript, 2008.

\bibitem{MS02}
Peter McMullen and Egon Schulte.
\newblock {\em Abstract Regular Polytopes}.
\newblock Cambridge University Press, 2002.
\newblock URL: \url{https://books.google.co.uk/books?id=JfmlMYe6MJgC}.

\bibitem{R07}
Alexander~A. Razborov.
\newblock Flag algebras.
\newblock {\em J. Symbolic Logic}, 72(4):1239--1282, 2007.
\newblock \href {https://doi.org/10.2178/jsl/1203350785} {\path{doi:10.2178/jsl/1203350785}}.

\bibitem{R08}
Alexander~A. Razborov.
\newblock On the minimal density of triangles in graphs.
\newblock {\em Combin. Probab. Comput.}, 17(4):603--618, 2008.
\newblock \href {https://doi.org/10.1017/S0963548308009085} {\path{doi:10.1017/S0963548308009085}}.

\bibitem{ST20}
Benny Sudakov and Istv\'an Tomon.
\newblock Tur\'an number of bipartite graphs with no {$K_{t,t}$}.
\newblock {\em Proc. Amer. Math. Soc.}, 148(7):2811--2818, 2020.
\newblock \href {https://doi.org/10.1090/proc/15042} {\path{doi:10.1090/proc/15042}}.

\end{thebibliography}

\end{document}